\documentclass[11pt,a4paper]{article}
\usepackage{amsthm, amsmath, amssymb, amsfonts, url, booktabs, tikz, setspace, fancyhdr, bm,lineno}
\usepackage[hidelinks]{hyperref}
\usepackage{geometry}
\usepackage{enumerate}
\usepackage[shortlabels]{enumitem}
\usepackage[babel]{microtype}
\usepackage[english]{babel}
\usepackage[capitalise]{cleveref}
\usepackage{comment}
\usepackage{bbm}
\usepackage{csquotes}
\usepackage{mathabx}

\counterwithin{figure}{section}

\newtheorem{theorem}{Theorem}[section]
\newtheorem{prop}[theorem]{Proposition}
\newtheorem{lemma}[theorem]{Lemma}
\newtheorem{cor}[theorem]{Corollary}
\newtheorem{conj}[theorem]{Conjecture}
\newtheorem{prob}[theorem]{Problem}
\newtheorem{claim}[theorem]{Claim}

\newtheorem{fact}[theorem]{Fact}

\crefname{theorem}{Theorem}{Theorems}
\Crefname{theorem}{Theorem}{Theorems}
\crefname{prop}{Proposition}{Propositions}
\Crefname{prop}{Proposition}{Propositions}
\crefname{lemma}{Lemma}{Lemmas}
\Crefname{lemma}{Lemma}{Lemmas}
\crefname{cor}{Corollary}{Corollaries}
\Crefname{cor}{Corollary}{Corollaries}
\crefname{conj}{Conjecture}{Conjectures}
\Crefname{conj}{Conjecture}{Conjectures}
\crefname{prob}{Problem}{Problems}
\Crefname{prob}{Problem}{Problems}
\crefname{claim}{Claim}{Claims}
\Crefname{claim}{Claim}{Claims}
\crefname{fact}{Fact}{Facts}
\Crefname{fact}{Fact}{Facts}

\theoremstyle{definition}
\newtheorem{defn}[theorem]{Definition}

\theoremstyle{remark}
\newtheorem*{remark}{Remark}

\newcommand{\one}{\mathbf{1}}
\newcommand{\Cov}{\mathrm{Cov}}
\newcommand{\Var}{\mathrm{Var}}
\newcommand{\Inf}{\mathrm{Inf}}
\newcommand{\I}{\mathcal{I}}
\newcommand\tup[1]{\left\langle #1 \right\rangle}

\title{Talagrand-Type Correlation Inequalities for Submodular and Supermodular Functions on the Hypercube}

\author{Fan Chang\thanks{School of Statistics and Data Science, Nankai University, Tianjin, China; and Extremal Combinatorics and Probability Group, Institute for Basic Science, Daejeon, South Korea. Email: \texttt{1120230060@mail.nankai.edu.cn}. Supported by the National Natural Science Foundation of China (NSFC) under grant 124B2019 and by the Institute for Basic Science (IBS-R029-C4).}, \quad 
Yu Chen\thanks{School of Mathematics and Statistics, Beijing Institute of Technology, Beijing 102488, China. Email: \texttt{yu.chen2023@bit.edu.cn}.}
}

\date{}  

\begin{document}
\maketitle

\begin{abstract}
Talagrand’s correlation inequality~\cite{Talagrand1996correlated} provides quantitative lower bounds on the covariance of two increasing Boolean functions in terms of their coordinate influences, but, in general, a logarithmic loss is necessary. Motivated by a question of Kalai, Keller and Mossel~\cite[Problem 6.1]{KKM2016correlation}, we identify a natural log-free regime. We prove that if two increasing Boolean functions on $\{0,1\}^n$ are either both submodular or both supermodular, then
\[
\mathbb{E}[fg]-\mathbb{E}[f]\mathbb{E}[g]\ge \frac{1}{4}\cdot\sum\limits_{i=1}^n{\rm Inf}_i[f]{\rm Inf}_i[g],
\]
where the constant $1/4$ is optimal. We also prove a real-valued extension: for two functions with the same second-difference sign, the covariance is bounded below by the sum of products of their Level-1 Fourier coefficients. As a consequence, we
verify the Friedgut--Kahn--Kalai--Keller spectral conjecture~\cite[Conjecture 5.8]{FKKK2018correlation} in this structured setting. The proofs combine a heat-semigroup representation based on second-order discrete derivatives with an independent induction argument for the Boolean case.
\vspace{10pt}

\noindent\textbf{Keywords:} Correlation inequalities, Influences, Submodular, Semigroup, Discrete Fourier analysis

\vspace{10pt}

\noindent\textbf{Mathematics Subject Classification:} 06E30; 60C05  
\end{abstract}

\section{Introduction}
We work throughout on the discrete hypercube $\{0,1\}^n$ equipped with the uniform product measure $\mu=(\frac{1}{2}\delta_1+\frac{1}{2}\delta_0)^{\otimes n}$. Unless stated otherwise, all expectations and probabilities are taken with respect to $\mu$; we write $\mathbb{E}$ and $\mathbb{P}$ for $\mathbb{E}_\mu$ and $\mathbb{P}_\mu$, respectively. We also write $\Cov(f,g):=\mathbb{E}[fg]-\mathbb{E}[f]\mathbb{E}[g]$. We order vectors coordinatewise, writing $x\le y$ if $x_i\le y_i$ for every $i\in[n]$.
\begin{defn}
A function $f:\{0,1\}^n\to\mathbb{R}$ is \emph{increasing} if for all $x,y\in\{0,1\}^n$,
\begin{equation*}
x\le y \Rightarrow f(x)\le f(y).
\end{equation*}
A set $A\subseteq\{0,1\}^n$ is called increasing if its indicator function $\mathbbm{1}_A$ is increasing, i.e., $x\in A$ and $x\le y$ imply $y\in A$.
\end{defn}
We recall the problem posed by Kalai, Keller, and Mossel~\cite[Problem 6.1]{KKM2016correlation}, which asks for a quantitative strengthening of the classical Harris--Kleitman correlation inequality~\cite{Harris1960,Kleitman1966}. For $i\in[n]$, let $e_i$ denote the $i$-th standard basis vector and write $x\oplus e_i$ for the point obtained from $x$ by flipping the $i$-th coordinate. The (uniform) \emph{influence of coordinate $i$} on a Boolean function $f:\{0,1\}^n\to\{0,1\}$ is
\[
{\rm Inf}_i[f]:=\mathbb{P}\left[f(x)\neq f(x\oplus e_i)\right],
\]
and the \emph{total influence} of $f$ is ${\rm I}[f]:=\sum_{i=1}^n{\rm Inf}_i[f]$.
\begin{prob}[Kalai--Keller--Mossel~\cite{KKM2016correlation}]
For any two increasing Boolean functions $f,g:\{0,1\}^n\to\{0,1\}$, identify additional conditions under which the following inequality holds:
\begin{equation}\label{ineq:Dream correlation}
    \Cov(f,g)\ge c\cdot \sum\limits_{i=1}^n{\rm Inf}_i[f]{\rm Inf}_i[g],
\end{equation}
where $c>0$ is a universal constant.
\end{prob}
Before this work, the only setting in which an inequality of the form \eqref{ineq:Dream correlation} was known to hold was an \emph{average-case} regime. Keller~\cite{Keller09correlation} showed that \eqref{ineq:Dream correlation} holds on average with $c=\frac{1}{4}$, meaning that the covariance is averaged over all pairs in a family $\mathcal{T}$ of increasing Boolean functions. Formally,
\begin{equation}\label{ineq:dream ineq-average}
    \sum\limits_{f,g\in\mathcal{T}}\Cov(f,g)\ge \frac{1}{4}\cdot\sum\limits_{f,g\in\mathcal{T}}\sum\limits_{i=1}^n{\rm Inf}_i[f]{\rm Inf}_i[g].
\end{equation}
Motivated by the suggestion of Kalai--Keller--Mossel~\cite{KKM2016correlation} that \emph{submodularity} may be the relevant structural hypothesis, we confirm this prediction in the pointwise (non-averaged) setting.

\begin{defn}[Submodularity/supermodularity]
For $x,y\in\{0,1\}^n$, set $(x\wedge y)_i:=\min\{x_i,y_i\}$ and $(x\vee y)_i:=\max\{x_i,y_i\}$.
A function $f:\{0,1\}^n\to\mathbb{R}$ is \emph{submodular} if
\[
f(x)+f(y) \ge f(x\wedge y)+f(x\vee y)\quad\text{for all}\ x,y\in\{0,1\}^n,
\]
and \emph{supermodular} if the reverse inequality holds for all $x,y$.
\end{defn}  

\begin{theorem}\label{thm:main thm-1}
Let $f,g:\{0,1\}^n\to\{0,1\}$ be increasing. If $f$ and $g$ are both supermodular, or both submodular, then
\begin{equation}\label{ineq:main thm-1}
   \Cov(f,g)\ge \frac{1}{4}\cdot\sum\limits_{i=1}^n{\rm Inf}_i[f]{\rm Inf}_i[g].
\end{equation}
\end{theorem}
The constant is sharp already for $n=1$ by taking $f=g=\mathbbm{1}_{\{1\}}$. We prove \cref{thm:main thm-1} in two complementary ways: by induction on the dimension using restrictions, and by an analytic approach based on the Bonami--Beckner semigroup on the hypercube~\cite{Beckner1975,Bonami1970}. In particular, the semigroup approach yields the following real-valued extension, which does not require monotonicity. Throughout the paper we use the Fourier–Walsh expansion on $\{0,1\}^n$ with respect to the orthonormal characters $\chi_S(x):=(-1)^{\sum_{i\in S}x_i}$, and write $\hat f(S):=\mathbb{E}[f(x)\chi_S(x)]$ for the corresponding Fourier coefficients.

\begin{theorem}\label{thm:main thm-2}
Let $f,g:\{0,1\}^n\to\mathbb{R}$. If $f$ and $g$ are both supermodular, or both submodular, then
\begin{equation}\label{ineq:level-one-lower-bound}
\Cov(f,g)\ge \sum_{i=1}^n \hat f(\{i\})\hat g(\{i\}).
\end{equation}
\end{theorem}

For increasing functions, the Level-$1$ Fourier coefficients in \eqref{ineq:level-one-lower-bound}
can be equivalently written in terms of one-sided $L^1$-influences.

\begin{cor}\label{cor:L1-influence-main}
Let $f,g:\{0,1\}^n\to\mathbb{R}$ be increasing. If $f$ and $g$ are both supermodular, or both submodular, then
\begin{equation}\label{ineq:L1influence1}
\Cov(f,g)\ge \frac{1}{4}\sum_{i=1}^n \mathrm{Inf}^{(1)}_i[f]\mathrm{Inf}^{(1)}_i[g],
\end{equation}
where
\[
\mathrm{Inf}^{(1)}_i[h]:=\mathbb{E}\left[\left|h\left(x^{(i\to1)}\right)-h\left(x^{(i\to0)}\right)\right|\right],
\]
and $x^{(i\to b)}$ denotes the vector obtained from $x\in\{0,1\}^{[n]\setminus\{i\}}$ by inserting $b\in\{0,1\}$ in the $i$-th coordinate. In particular, when $f,g$ are $\{0,1\}$-valued, Corollary~\ref{cor:L1-influence-main} gives \cref{thm:main thm-1}.
\end{cor}

Friedgut--Kahn--Kalai--Keller~\cite[Conjecture 5.8]{FKKK2018correlation} also proposed a Fourier-analytic strengthening of Harris--Kleitman's inequality in the form of the following conjecture.

\begin{conj}[Friedgut--Kahn--Kalai--Keller~\cite{FKKK2018correlation}]\label{conj:spectral-FKKK}
For any increasing Boolean functions $f,g:\{0,1\}^n\to\{0,1\}$,
\begin{equation}\label{eq:spectral-conj}
\Cov(f,g)\ge 4\sum_{S\neq\emptyset}|S|\hat f(S)^2\hat g(S)^2.
\end{equation}
\end{conj}

This conjecture is weaker than the desired relation~\eqref{ineq:Dream correlation}.  Indeed, for Boolean functions,
\begin{equation}\label{eq:spectral-weaker-than-dream}
4\sum_{S\neq\emptyset}|S|\hat f(S)^2\hat g(S)^2
=4\sum_{i=1}^n\sum_{S\ni i}\hat f(S)^2\hat g(S)^2
\le
\frac14\sum_{i=1}^n\Inf_i[f]\Inf_i[g],
\end{equation}
where the last step follows from the identity
$\Inf_i[h]=4\sum_{S\ni i}\hat h(S)^2$.  Consequently, \cref{thm:main thm-1} immediately yields the following structured case of Conjecture~\ref{conj:spectral-FKKK}.
\begin{cor}\label{cor:supermodular-submodular}
Let $f,g:\{0,1\}^n\to\{0,1\}$ be increasing. If $f$ and $g$ are both supermodular or both submodular, then
\[
\Cov(f,g)\ge 4\sum_{S\neq\emptyset}|S|\hat f(S)^2\hat g(S)^2.
\]
\end{cor}

Although Corollary~\ref{cor:supermodular-submodular} follows directly from~\cref{thm:main thm-1}, we give a separate proof in~\cref{subsec:spectral-FKKK}. That proof isolates the difference between covariance and the spectral quantity in~\eqref{eq:spectral-conj} through a first-order semigroup defect formula. The resulting auxiliary lower bounds for the defect may be of independent interest.

\begin{remark}
Our proof hinges on the analytic characterization of submodularity/supermodularity via second differences. For distinct $i,j\in[n]$, define
\[
\partial_{ij}f(x):=\partial_i(\partial_j f)(x)
=f\left(x^{(i\to 1,j\to 1)}\right)-f\left(x^{(i\to 1,j\to 0)}\right)-f\left(x^{(i\to 0,j\to 1)}\right)+f\left(x^{(i\to 0,j\to 0)}\right).
\]
Submodularity is equivalent to $\partial_{ij}f\le 0$ pointwise (and supermodularity to $\partial_{ij}f\ge 0$), see Fact~\ref{lem:equiv-submod}. Under the structural assumptions $\partial_{ij}f,\partial_{ij}g\ge 0$ or $\partial_{ij}f,\partial_{ij}g\le 0$, Borell's reverse hypercontractivity~\cite{Borel1982,MORSS2006} yields the following strengthened covariance lower bound:
\begin{equation}\label{ineq: main strong lower}
\Cov(f,g)\ge\sum_{i=1}^n\hat{f}(\{i\})\hat{g}(\{i\})+ c(\theta)\sum_{1\le i<j\le n}\|\partial_{ij}f\|_{1-\theta}\|\partial_{ij}g\|_{1-\theta},
\end{equation}
where $c(\theta):=\frac{\theta-\theta^2/2}{8}\in(0,1)$ for $\theta\in(0,1)$.
Since
\[
\Cov(f,g)-\sum_{i=1}^n\hat f(\{i\})\hat g(\{i\})=\sum_{|S|\ge 2}\hat f(S)\hat g(S),
\]
the regime $\sum_{|S|\ge 2}\hat f(S)\hat g(S)=0$ is especially informative: \eqref{ineq: main strong lower} then forces vanishing of each weighted second-difference term, suggesting additional structural constraints on $f,g$. We return to this point in~\cref{sec:reverse}.
\end{remark}
We now turn from lower bounds to \emph{upper} bounds on correlation for arbitrary real-valued functions with no monotonicity assumption. In a recent note, Mossel~\cite[Claim 5.17]{Mossel2020} established the following two-function version of the Poincar\'e inequality. 
\begin{theorem}[Mossel’s two-function Poincar\'e inequality]
Let $f,g:\{0,1\}^n\to\mathbb{R}$. Then
\begin{equation}
\left|\Cov(f,g)\right|\le \frac{1}{4}\sum\limits_{i=1}^n\sqrt{{\rm Inf}_i[f]{\rm Inf}_i[g]}.
\end{equation}
\end{theorem}
We refine this inequality by separating the Fourier terms whose supports lie inside a prescribed set of coordinates. This gives an iterated refinement of Mossel's bound.

\begin{theorem}\label{thm:stronger two poincare}
Let $f,g:\{0,1\}^n\to\mathbb{R}$. Then for every $J\subseteq[n]$,
\begin{equation}\label{eq:iterated-level-one-refinement}
\left|\Cov(f,g)\right|\le\sum_{\emptyset\ne A\subseteq J}|\hat f(A)\hat g(A)|+\frac14\sum_{j\in[n]\setminus J}\sqrt{{\rm Inf}_j[f]{\rm Inf}_j[g]}.
\end{equation}
\end{theorem}

\begin{remark}
Taking $J=\emptyset$ recovers Mossel's two-function Poincar\'e inequality. More generally, \cref{thm:stronger two poincare} is stronger than Mossel's original bound. Indeed,
\[
\begin{aligned}
\sum_{\emptyset\ne A\subseteq J}|\hat f(A)\hat g(A)|&\le\sum_{j\in J}\sum_{\substack{A\subseteq J\\ j\in A}}|\hat f(A)\hat g(A)|\le
\sum_{j\in J}\sqrt{\left(\sum_{S:j\in S}\hat f(S)^2\right)
\left(\sum_{S:j\in S}\hat g(S)^2\right)}\\
&=\frac14\sum_{j\in J}
\sqrt{{\rm Inf}_j[f]{\rm Inf}_j[g]}.
\end{aligned}
\]
Thus the right-hand side of \eqref{eq:iterated-level-one-refinement} is bounded above by $\frac14\sum_{j=1}^n\sqrt{{\rm Inf}_j[f]{\rm Inf}_j[g]}$. In particular, taking $J=\{i\}$ gives the fixed-coordinate refinement
\[
\left|\Cov(f,g)\right|
\le
|\hat f(\{i\})\hat g(\{i\})|
+
\frac14\sum_{j\in[n]\setminus\{i\}}
\sqrt{{\rm Inf}_j[f]{\rm Inf}_j[g]}.
\]
\end{remark}

Comparing to the lower bound \eqref{ineq:L1influence1}, we can quantify the deviation of the covariance from its Level-1 contribution for \emph{general} increasing functions. In particular, we obtain an $L^1$-$L^2$-type upper bound of Talagrand flavor~\cite{talagrand1994russo}.
\begin{theorem}[Talagrand $L^1$-$L^2$-type upper bound]\label{thm:main-thm4}
Let $f,g:\{0,1\}^n\to\mathbb{R}$ be increasing. Then
\begin{equation}
\left|
\Cov(f,g)-\frac{1}{4}\sum_{i=1}^n {\rm Inf}^{(1)}_i[f]{\rm Inf}^{(1)}_i[g]
\right|
\le
\frac{9}{8}\sum_{1\le i<j\le n}
\frac{\|\partial_{ij} f\|_{2}\,\|\partial_{ij} g\|_{2}}
{\,1+\log\!\left(\dfrac{\|\partial_{ij} f\|_{2}\,\|\partial_{ij} g\|_{2}}{\|\partial_{ij} f\|_{1}\,\|\partial_{ij} g\|_{1}}\right)}.
\end{equation}
A summand is interpreted as $0$ if $\partial_{ij}f\equiv0$ or $\partial_{ij}g\equiv0$.
\end{theorem}

\subsection{Related Work}
A classical starting point is the Harris--Kleitman correlation inequality~\cite{Harris1960,Kleitman1966}, which asserts $\mathbb{E}[fg]\ge \mathbb{E}[f]\mathbb{E}[g]$ for increasing Boolean functions $f,g$. In a program to \emph{quantify} this positive correlation, Talagrand~\cite{Talagrand1996correlated} proposed to measure the simultaneous dependence of $f$ and $g$ on coordinates via the \emph{cross-total-influence of $f,g$}
\[
\I[f,g]:=\sum_{i=1}^n \mathrm{Inf}_i[f]\mathrm{Inf}_i[g],
\]
and asked whether one can lower bound the covariance by $\I[f,g]$. This leads to what we call a \emph{Talagrand-type correlation inequality} for increasing Boolean functions:
\begin{equation}\label{eq:tal-type}
\Cov(f,g)\stackrel{?}{\ge}C\cdot\I[f,g],
\end{equation}
where $C>0$ is a universal constant. Without additional structure, however, the inequality~\eqref{eq:tal-type} fails for general increasing pairs; one therefore seeks natural hypotheses under which it holds. 

The first positive evidence came in an \emph{average-case} form: Keller~\cite{Keller09correlation} proved that the analog of \eqref{ineq:Dream correlation} holds after averaging over all pairs in a family $\mathcal{T}$ of increasing Boolean functions; see~\eqref{ineq:dream ineq-average}. A second, conceptually different, reduction isolates the \emph{antipodal} condition $g(x)=1-g(1-x)$, which forces $\mathbb{E}[g]=\tfrac12$ and links the problem to extremal set theory: Friedgut, Kahn, Kalai, and Keller~\cite{FKKK2018correlation} showed that the following conjecture is equivalent to the celebrated Chv\'atal conjecture~\cite{Chvatal1974}:
\begin{conj}[Friedgut--Kahn--Kalai--Keller~\cite{FKKK2018correlation}]\label{conj:FKKK}
If $f,g:\{0,1\}^n\to\{0,1\}$ are increasing and $g$ is antipodal, then
\begin{equation}\label{ineq:Chvatal}
\Cov(f,g) \ge\tfrac{1}{4}\cdot \min_{i\in[n]} \mathrm{Inf}_i[f].
\end{equation}
\end{conj}
Moreover, a standard application of Harper's edge-isoperimetric inequality (${\rm I}[g]\ge 2\alpha\log_2(1/\alpha)$ for $\alpha=\mathbb{E}[g]\le\tfrac12$) shows that a cross-total-influence lower bound of the form 
$$
\Cov(f,g)\ge \tfrac14\I[f,g]
$$ 
already implies a Chv\'atal-type estimate $\Cov(f,g)\ge \frac{\alpha}{2}\log_2(\frac{1}{\alpha})\cdot \min_i{\rm Inf}_i[f]$, and hence $\Cov(f,g)\ge \tfrac14\min_i{\rm Inf}_i[f]$ when $g$ is antipodal. Consequently, when $f,g$ are additionally submodular or supermodular, our bound~\eqref{ineq:main thm-1} immediately yields the Chv\'atal-type correlation estimate in this structured regime by the preceding Harper-based reduction. 

In full generality, the best universal lower bound is Talagrand’s celebrated inequality~\cite{Talagrand1996correlated},
\[
\Cov(f,g) \ge C\cdot\frac{\I[f,g]}{\log\left(e/\I[f,g]\right)},
\]
where $C>0$ is universal. The logarithmic loss is known to be unavoidable. Tightness is witnessed by several natural families, including small Hamming balls versus their duals~\cite{Talagrand1996correlated}, Tribes and dual Tribes~\cite{Keller09correlation}, and halfspaces and their duals~\cite[Corollary~1.2]{KK2019DA}. Kalai--Keller--Mossel~\cite{KKM2016correlation} further gave a necessary condition and two sufficient conditions for tightness. Talagrand's original proof proceeds by induction on coordinates and a Level-$1:2$ Fourier inequality; Kalai--Keller--Mossel~\cite{KKM2016correlation} later provided a semigroup interpolation proof that avoids induction while retaining the same level-structure input. A complementary inequality due to Keller–Mossel–Sen~\cite{KMS2014correlation} replaces the single global cross term by coordinatewise contributions,
\[
\Cov(f,g) \ge C\cdot\sum_{i=1}^n \frac{{\rm Inf}_i[f]}{\sqrt{\log(e/{\rm Inf}_i[f])}}\cdot \frac{{\rm Inf}_i[g]}{\sqrt{\log(e/{\rm Inf}_i[g])}},
\]
which, in specific regimes (e.g., a small Hamming ball against the majority), improves Talagrand’s bound~\cite{KKM2016correlation}. The proof in \cite{KMS2014correlation} proceeds via a Gaussian analog and Borell’s reverse isoperimetry, and it can also be recovered by adapting Talagrand’s inductive scheme~\cite[Remark~4.6]{KKM2016correlation}.

More recently, Eldan~\cite{Eldan22} improved the logarithmic factor under the antipodality of one function:
\begin{theorem}[Eldan~\cite{Eldan22}]
If $f,g:\{0,1\}^n\to\{0,1\}$ are increasing and $g$ is antipodal, then
\[
\Cov(f,g) \ge C\cdot \frac{\I[f,g]}{\sqrt{\log\left(2e/\I[f,g]\right)}}.
\]
\end{theorem}
The proof proceeds via a Gaussian counterpart, representing the correlation as a quadratic covariation of a suitable stochastic process and invoking Fourier comparison between Level-1 and Level-3 contributions; a discretization step then yields the Boolean result.

In the present paper, we establish  \cref{thm:main thm-1,thm:main thm-2} for submodular and supermodular functions by exploiting their analytic properties--most notably, that submodularity of a set function is equivalent to all second-order discrete partial derivatives being nonpositive, with the inequalities reversed for supermodularity (see~\cref{lem:equiv-submod}). Fourier-analytic methods for submodular and supermodular functions have a substantial history in theoretical computer science. Motivated by the learning of submodular functions and their applications to differential privacy~\cite{GHRU2013}, Cheraghchi, Klivans, Kothari, and Lee~\cite{CKKL} showed that every submodular function can be $\varepsilon$-approximated in the $\ell_2$-norm by a polynomial of degree $O(1/\varepsilon^2)$. Their proof analyzes the noise sensitivity of submodular functions, a standard Fourier-analytic tool for establishing spectral concentration of low degree. Subsequently, Feldman, Kothari, and Vondr\'ak~\cite{FKV2013} obtained the same $O(1/\varepsilon^2)$ upper bound via approximation by real-valued decision trees. Feldman and Vondr\'ak~\cite{FV2016junta} studied approximation of submodular, XOS (fractionally subadditive), and self-bounding functions by juntas. They later derived tight bounds on the polynomial degree sufficient to approximate any function in these classes in $\ell_2$-norm~\cite{FV2015}. Finally, Feldman, Kothari, and Vondr\'ak~\cite{FKV2020} provided nearly tight bounds for approximating self-bounding functions (including submodular and XOS functions) by low-degree polynomials and juntas in the $\ell_1$-norm, obtained via a noise-stability analysis.

\medskip\noindent\emph{Organization.} This paper is organized as follows. \cref{sec:prelim} collects preliminaries on Fourier analysis over the hypercube, discrete derivatives, the heat semigroup, and submodularity/supermodularity. \cref{sec:lower} develops our \emph{lower} bounds for correlation. We first give a fully discrete proof of the Boolean case (\cref{thm:main thm-1}) via induction on the dimension. We then derive a heat-semigroup identity isolating Level-$\ge2$ Fourier weight and use it to prove the real-valued super/submodular correlation inequality (\cref{thm:main thm-2}). Finally, we prove Corollary~\ref{cor:supermodular-submodular} through a first-order semigroup defect identity in~\cref{subsec:spectral-FKKK}. \cref{sec:upper} turns to \emph{upper} bounds. We establish a Talagrand-type $L^1$-$L^2$ estimate that controls the Level-$\ge2$ contribution through second-order discrete derivatives (\cref{thm:main-thm4}), and recall Mossel's two-function Poincar\'e inequality together with an iterated Fourier-support refinement (\cref{thm:stronger two poincare}). Finally, \cref{sec:reverse} revisits reverse hypercontractivity: under the analytic sub/supermodular condition, we combine a heat-semigroup representation with Borell's reverse hypercontractivity to obtain the strengthened lower bound \eqref{ineq: main strong lower}, and then explain the sign obstruction that remains without second-difference sign information.

\section{Preliminaries}\label{sec:prelim}
\subsection{Fourier Analysis on the Hypercube}
We consider real-valued functions $f:\{0,1\}^n \to \mathbb{R}$, equipped with the inner product $\tup{f,g}=\mathbb{E}_x[f(x)g(x)]$. As above, $\Cov(f,g)$ denotes covariance. For $S\subseteq[n]$, define the Fourier--Walsh character $\chi_S(x):=(-1)^{\sum_{i\in S}x_i}$. The family $
\{\chi_S\}_{S\subseteq[n]}$ is an orthonormal basis of $L^2(\{0,1\}^n)$. The Fourier--Walsh expansion of $f$ is given by $f(x)=\sum_{S\subseteq[n]}\hat{f}(S)\chi_S(x)$, where $\hat{f}(S)=\tup{f,\chi_S}$. For $p>0$, we write $\|f\|_p:=\left(\mathbb{E}_x [|f(x)|^p]\right)^{1/p}$, with the expectation taken over the relevant cube. This is the usual $L^p$-norm for $p\ge1$, and the standard $L^p$ quasi-norm for $0<p<1$.

We next introduce the discrete derivatives used throughout the paper.
\begin{defn}
The $i$-th discrete derivative $\partial_i f$ is the function on $\{0,1\}^{[n]\setminus\{i\}}$ defined by
\[
\partial_i f(x)=f(x^{(i\to 1)})-f(x^{(i\to 0)}),
\]
where $x^{(i\to b)}\in\{0,1\}^n$ is obtained from $x\in\{0,1\}^{[n]\setminus\{i\}}$ by inserting $b\in\{0,1\}$ in the $i$-th coordinate.
\end{defn}

\begin{fact}\label{fact:fourier expansion of derivative}
Let $f:\{0,1\}^n\to\mathbb{R}$, and let $i\in[n]$. Then
\[
\partial_i f(x)=-2\sum_{S:i\in S}\hat{f}(S)\chi_{S\setminus\{i\}}(x).
\]
In particular, $\mathbb{E}_x[\partial_i f(x)]=-2\hat{f}(\{i\})$.
\end{fact}
\begin{proof}
By the Fourier expansion and linearity,
\[
\partial_i f(x)=\sum_{S\subseteq[n]}\hat f(S)\left(\chi_S(x^{(i\to 1)})-\chi_S(x^{(i\to 0)})\right).
\]
If $i\notin S$, the two terms cancel. If $i\in S$, then $\chi_S(x^{(i\to 0)})=\chi_{S\setminus\{i\}}(x)$ and $\chi_S(x^{(i\to 1)})=-\chi_{S\setminus\{i\}}(x)$, giving the coefficient $-2\chi_{S\setminus\{i\}}(x)$. Taking expectations gives
\[
\mathbb{E}_x[\partial_i f(x)]
=-2\sum_{S:i\in S}\hat f(S)\mathbb{E}[\chi_{S\setminus\{i\}}]
=-2\hat f(\{i\}).\tag*{\qedhere}
\]
\end{proof}
Discrete derivative operators are handy for our problems. A function $f:\{0,1\}^n\to\mathbb{R}$ is increasing if and only if $\partial_i f(x)\ge 0$ for every $i\in[n]$ and every $x\in\{0,1\}^{[n]\setminus\{i\}}$. Since
discrete derivatives in distinct coordinates commute, we may compose them
without ambiguity. For distinct $i,j\in[n]$, let $\partial_{ij}=\partial_i\circ\partial_j$. Explicitly, for $x\in\{0,1\}^{[n]\setminus\{i,j\}}$,
\[
\partial_{ij}f(x)=f\left(x^{(i\to 1,j\to 1)}\right)-f\left(x^{(i\to 1,j\to 0)}\right)-f\left(x^{(i\to 0,j\to 1)}\right)+f\left(x^{(i\to 0,j\to 0)}\right).
\]
\begin{fact}
Let $f:\{0,1\}^n\to\mathbb{R}$, and let $i\neq j\in[n]$. Then
\[
\partial_{ij}f(x)=4\sum_{S:i,j\in S}\hat{f}(S)\chi_{S\setminus\{i,j\}}(x).
\]
\end{fact}
\begin{proof}
This follows by applying Fact~\ref{fact:fourier expansion of derivative} twice; the
two factors of $-2$ give the coefficient $4$.
\end{proof}
Recent work has studied such higher-order operators; see~\cite{oleszkiewicz2023boolean,przybylowski2024kkl,tanguy2020talagrand}.

Discrete derivatives provide a convenient formulation of influence. For $p>0$, define
\begin{equation}
    {\rm Inf}_i^{(p)}[f]=\mathbb{E}\left[|\partial_i f|^p\right], \qquad {\rm I}^{(p)}[f]=\sum_{i=1}^n{\rm Inf}_i^{(p)}[f].
\end{equation}
When no superscript is displayed, we write ${\rm Inf}_i[f]:={\rm Inf}_i^{(2)}[f]=\mathbb{E}[|\partial_i f|^2]$ for arbitrary real-valued $f$. For Boolean $f:\{0,1\}^n\to\{0,1\}$, Plancherel's identity together with Fact~\ref{fact:fourier expansion of derivative} give
\[
{\rm Inf}_i[f]=\mathbb{E}\left[\left|f\left(x^{(i\to 1)}\right)-f\left(x^{(i\to 0)}\right)\right|^2\right]
=4\sum_{S:i\in S}\hat{f}(S)^2.
\]
\begin{fact}\label{fact:Boolean and increasing}
If $f:\{0,1\}^n\to\{0,1\}$ is increasing, then ${\rm Inf}_i[f]=-2\hat{f}(\{i\})$.
\end{fact}
\begin{proof}
For increasing Boolean-valued $f$, we have $\partial_i f\in\{0,1\}$ pointwise, hence ${\rm Inf}_i[f]=\mathbb{E}[|\partial_i f|^2]=\mathbb{E}[\partial_i f]=-2\hat{f}(\{i\})$, where the last equality follows from Fact~\ref{fact:fourier expansion of derivative}.
\end{proof}

We now introduce the heat semigroup, also called the noise operator. It is our
main analytic tool and will be used repeatedly to relate smoothing to Fourier
levels and discrete derivatives. Define
\[
P_t=\left(e^{-t}{\rm id}+(1-e^{-t})\mathbb{E}\right)^{\otimes n}, \quad t\ge 0,
\]
where $\mathbb E$ in each tensor factor denotes averaging over one coordinate. Then $P_t$ is a self-adjoint semigroup of unital positive linear operators on $\mathbb{R}^{\{0,1\}^n}$. It satisfies $P_0f=f$ and $\lim_{t\to\infty}P_t f=\mathbb{E}[f]$. In particular, $P_t$ is order preserving: if $f\ge g$ pointwise, then $P_t f\ge P_t g$ pointwise. We use the same notation $P_t$ for the heat semigroup on lower-dimensional cubes obtained after taking discrete derivatives.

\begin{fact}\label{fact:heat-fourier}
Let $f:\{0,1\}^n\to\mathbb{R}$, and let $t\ge 0$. Then
\[
P_t f(x)=\sum_{S\subseteq[n]}e^{-t|S|}\hat{f}(S)\chi_S(x).
\]
\end{fact}
Finally, we recall the Bonami--Beckner hypercontractive inequality on the hypercube~\cite{Beckner1975,Bonami1970}, which we will invoke repeatedly in what follows.
\begin{theorem}\label{thm:hyperc}
Let $f:\{0,1\}^n\to\mathbb R$, and let $t\ge 0$. Then
\begin{equation}
    \|P_t f\|_2\le \|f\|_{1+e^{-2t}}.
\end{equation}
\end{theorem}

\subsection{Submodular and Supermodular Functions}

We identify subsets $A\subseteq[n]$ with their indicator vectors $\one_A\in\{0,1\}^n$, where $(\one_A)_i=1$ if and only if $i\in A$. Under this identification, we write $f(A)$ for $f(\one_A)$. A set function
$f:2^{[n]}\to\mathbb R$ is \emph{submodular} if $f(A\cup B)+f(A\cap B)\le f(A)+f(B)$ for all $A,B\subseteq[n]$, and \emph{supermodular} if the reverse inequality holds.

Equivalently, for functions on the cube, submodularity can be expressed in terms
of the lattice operations
\[
(x\wedge y)_i:=\min\{x_i,y_i\},
\qquad
(x\vee y)_i:=\max\{x_i,y_i\}.
\]
\begin{fact}\label{lem:equiv-submod}
Let $f:\{0,1\}^n\to\mathbb{R}$. The following conditions are equivalent:
\begin{itemize}
    \item[(1)] (Submodularity) For all $x,y\in\{0,1\}^n$, $f(x)+f(y) \ge f(x\wedge y)+f(x\vee y)$.
    \item[(2)] (Mixed second differences are non-positive) For all $x$ and all $i\neq j$, $\partial_{ij}f(x) \le 0$.
    \item[(3)] (Diminishing returns) For all $A\subseteq B\subseteq[n]$ and all $k\notin B$,
\[
f(A\cup\{k\})-f(A)\ \ge\ f(B\cup\{k\})-f(B).
\]
\end{itemize}
The same three statements with all inequalities reversed are equivalent and characterize supermodularity.
\end{fact}

\begin{proof}
For a fixed $x$ and distinct $i,j$, write $f(uv):= f(x^{(i\to u,j\to v)})$ for $u,v\in\{0,1\}$. Then $\partial_{ij}f(x)=f(11)-f(10)-f(01)+f(00)$. We prove the submodular case; the supermodular case follows by reversing all inequalities.

\emph{(1) $\Rightarrow$ (2).}
Fix $x$ and $i\neq j$. Applying (1) to $x^{(i\to0,j\to1)}$ and $x^{(i\to1,j\to0)}$ yields
\[
f(01)+f(10) \ge f(00)+f(11),
\]
which is equivalent to $\partial_{ij}f(x)=f(11)-f(10)-f(01)+f(00)\le 0$.

\emph{(2) $\Rightarrow$ (3).} Let $A\subseteq B\subseteq[n]$ and $k\notin B$. Choose an increasing chain
$A=A_0\subset A_1\subset\cdots\subset A_m=B$ with $A_t=A_{t-1}\cup\{j_t\}$.
For each $t$, applying (2) at the base point $\one_{A_{t-1}}$ and to the
pair $(k,j_t)$ gives
\[
\begin{aligned}
0 \ge \partial_{k,j_t}f(\one_{A_{t-1}})
&= f(A_{t-1}\cup\{k,j_t\})-f(A_{t-1}\cup\{j_t\})-f(A_{t-1}\cup\{k\})+f(A_{t-1})\\
&=\left[f(A_{t}\cup\{k\})-f(A_t)\right]-\left[f(A_{t-1}\cup\{k\})-f(A_{t-1})\right].
\end{aligned}
\]
Thus $f(A_{t-1}\cup\{k\})-f(A_{t-1})\ge f(A_t\cup\{k\})-f(A_t)$ for every $t$; chaining over $t=1,\dots,m$ gives (3).

\emph{(3) $\Rightarrow$ (1).}
Let $X,Y\subseteq[n]$, set $A:=X\cap Y$, and choose a chain
$A=B_0\subset B_1\subset\cdots\subset B_m=Y$ with $B_t=B_{t-1}\cup\{j_t\}$.
Since $j_t\in Y\setminus X$, we have $j_t\notin X\cup B_{t-1}$. Applying
(3) with $B_{t-1}\subseteq X\cup B_{t-1}$ and $k=j_t$, we obtain
\[
f(B_t)-f(B_{t-1}) \ge f(X\cup B_t)-f(X\cup B_{t-1}).
\]
Summing over $t$ gives $f(Y)-f(X\cap Y)\ge f(X\cup Y)-f(X)$, i.e. $f(X)+f(Y)\ge f(X\cup Y)+f(X\cap Y)$.
\end{proof}
Consequently, this fact establishes a precise bridge between (super/sub)modularity and discrete derivatives: a function $f:\{0,1\}^n\to\mathbb{R}$ is submodular/supermodular if and only if, for each $i\in[n]$, the discrete derivative $\partial_i f$ is pointwise non-increasing/non-decreasing in every other coordinate, or equivalently, for all $i\neq j$, $\partial_{ij}f\le 0$ ($\partial_{ij}f\ge 0$).

\section{Lower Bounds for Correlation}\label{sec:lower}
\subsection{An Induction Proof of~\cref{thm:main thm-1}}

We first present an inductive proof of~\cref{thm:main thm-1}. The restriction setup and the covariance decomposition below follow the standard induction-by-restrictions framework used by Kalai--Keller--Mossel~\cite[pp.~266--267]{KKM2016correlation}.

\begin{proof}[Proof of~\cref{thm:main thm-1} via induction]
Assume $f,g:\{0,1\}^n\to\{0,1\}$ are increasing and either both submodular ($\partial_{ij}\le0$) or both supermodular ($\partial_{ij}\ge0$).

The proof is by induction on $n$. The base case $n=1$ is immediate, since 
\[
\Cov(f,g)=\mathbb{E}[fg]-\mathbb{E}[f]\mathbb{E}[g]
=\sum_{S\neq\emptyset}\hat{f}(S)\hat{g}(S)=\hat{f}(\{1\})\hat{g}(\{1\})
=\frac{1}{4}{\rm Inf}_1[f]{\rm Inf}_1[g],
\]
where the last equality follows from Fact~\ref{fact:Boolean and increasing}.

For the induction step, write $x=(x_{-n},x_n)$ and define the restrictions $f^0,f^1:\{0,1\}^{n-1}\to\{0,1\}$ by
\[
f^0(x_{-n}):=f(x_{-n},0),\qquad f^1(x_{-n}):=f(x_{-n},1).
\]
Set $a^{\ell}=\mathbb{E}[f^{\ell}]$ for $\ell\in\{0,1\}$, $a_i:={\rm Inf}_i[f]$ for $i\in[n]$, and $a_i^{\ell}:={\rm Inf}_i[f^{\ell}]$ for $i\in[n-1]$ and $\ell\in\{0,1\}$. Define $g^0,g^1,b^{\ell},b_i,b_i^{\ell}$ in the same way, with $g$ in place of $f$.

We first record that the restrictions preserve the relevant structure.
\begin{claim}[Restriction preserves monotonicity and (super/sub)modularity]\label{lem:restriction-preserves}
Fix $b\in\{0,1\}$. If $f$ is increasing, then $f^b$ is increasing. If $f$ is submodular (resp. supermodular), then $f^b$ is submodular (resp. supermodular) on $\{0,1\}^{n-1}$.
\end{claim}
\begin{proof}
For $i\le n-1$, $\partial_i f^b$ is obtained by fixing the $n$th coordinate of $\partial_i f$ to $b$; hence $\partial_i f\ge0$ implies $\partial_i f^b\ge0$. For $i\ne j\le n-1$, $\partial_{ij} f^b$ is obtained by fixing the $n$th coordinate of $\partial_{ij} f$ to $b$. Thus the sign condition $\partial_{ij}f\le0$ (resp. $\ge0$) is inherited by $f^b$, and Fact~\ref{lem:equiv-submod} completes the proof.
\end{proof}
Thus, each restriction $f^b,g^b$ is increasing and inherits the same sub/supermodularity sign, so by the induction hypothesis in dimension $n-1$,
\[
\begin{aligned}
\Cov(f^1,g^1)&=\mathbb{E}[f^1g^1]-a^1b^1\ge \frac{1}{4} \sum_{i=1}^{n-1}a_i^1b_i^1,\\
\Cov(f^0,g^0)&=\mathbb{E}[f^0g^0]-a^0b^0\ge \frac{1}{4} \sum_{i=1}^{n-1}a_i^0b_i^0.
\end{aligned}
\]
Moreover, since $f$ is increasing and Boolean, $\partial_n f\in\{0,1\}$ pointwise and hence
\[
a^1-a^0=\mathbb{E}[f^1-f^0]=\mathbb{E}[\partial_n f]={\rm Inf}_n[f]=a_n,
\]
and similarly $b^1-b^0=b_n$. Decomposing the covariance along $x_n$ gives
\[
\begin{aligned}
\Cov(f,g)&=\frac{1}{2}(\mathbb{E}[f^1g^1]+\mathbb{E}[f^0g^0])-\frac{1}{4}(a^1+a^0)(b^1+b^0)\\
&=\frac{1}{2}\Cov(f^1,g^1)+\frac{1}{2}\Cov(f^0,g^0)+\frac{(a^1-a^0)(b^1-b^0)}{4}\\
&\ge \frac{1}{8}\sum_{i=1}^{n-1}\left(a_i^1b_i^1+a_i^0b_i^0\right)+\frac{a_nb_n}{4}.
\end{aligned}
\]
Thus, it remains to show
\[
\sum_{i=1}^{n-1}\left(a_i^1b_i^1+a_i^0b_i^0\right)
\ge 2\sum_{i=1}^{n-1}a_i b_i
=\frac{1}{2}\sum_{i=1}^{n-1}(a_i^1+a_i^0)(b_i^1+b_i^0),
\]
where we used $a_i=\frac{1}{2}(a_i^1+a_i^0)$ and $b_i=\frac{1}{2}(b_i^1+b_i^0)$. Equivalently, it suffices to prove
\[
\sum_{i=1}^{n-1}(a_i^1-a_i^0)(b_i^1-b_i^0)\ge0.
\]

\begin{claim}\label{claim:submodular-restriction}
If $f$ is submodular, then $\partial_i f^0\ge \partial_i f^1$ pointwise for every $i\in[n-1]$; if $f$ is supermodular, then the inequality is reversed. The same statements hold for $g$.
\end{claim}
\begin{proof}
For any fixed values of the remaining coordinates, submodularity on the $(i,n)$-square gives
\[
\partial_{in}f=f(11)-f(10)-f(01)+f(00)\le0,
\]
equivalently $f(10)-f(00)\ge f(11)-f(01)$. This says exactly that the $i$-th derivative at $x_n=0$ dominates the $i$-th derivative at $x_n=1$ pointwise. The supermodular case is analogous.
\end{proof}
By \cref{claim:submodular-restriction} and averaging, $a_i^1-a_i^0\le 0$ and $b_i^1-b_i^0\le 0$ in the submodular case for every $i\in[n-1]$, while both are $\ge 0$ in the supermodular case. Hence $(a_i^1-a_i^0)(b_i^1-b_i^0)\ge0$ in either case, completing the induction.
\end{proof}

\begin{remark}
The same argument extends to bounded increasing functions $f,g:\{0,1\}^n\to[-1,1]$ upon replacing the $L^2$-influence by the $L^1$-influence.
\end{remark}

\subsection{A Heat-semigroup Representation and the Proof of~\cref{thm:main thm-2}}
We next give the analytic proof of the real-valued result. The key is a heat-semigroup identity that isolates the Level-$\ge2$ Fourier weight via second-order discrete derivatives.
\begin{lemma}[Heat-semigroup representation with $\partial_{ij}$]\label{lem:representation_of_high_level_correlation}
For any $f,g:\{0,1\}^n\to\mathbb{R}$, we have
\begin{equation}\label{eq:heat-partial}
\sum_{|S|\ge 2}\hat f(S)\hat g(S)
=\frac{1}{8}\sum_{1\le i<j\le n}\int_0^\infty \left(1-e^{-t}\right)e^{-t}\,\mathbb{E}\left[\partial_{ij} f\cdot P_t\partial_{ij} g\right]dt.
\end{equation}
\end{lemma}
\begin{proof}
For $S\subseteq[n]$ with $i,j\in S$ one checks
$\partial_{ij}\chi_S=4\,\chi_{S\setminus\{i,j\}}$, hence
\[
\mathbb{E}[\partial_{ij} f\cdot P_t\partial_{ij} g]
=16\sum_{S\ni i,j} e^{-t(|S|-2)}\hat f(S)\hat g(S).
\]
Summing over $i<j$ contributes the factor $\binom{|S|}{2}$; the Laplace kernel
$\big(1-e^{-t}\big)e^{-t}$ satisfies
\[
\int_0^\infty \big(1-e^{-t}\big)e^{-t}\,e^{-t(|S|-2)}\,dt
=\frac{1}{(|S|-1)|S|}.
\]
Therefore the right-hand side of \eqref{eq:heat-partial} equals
\[
\frac{1}{8}\cdot 16 \sum_{|S|\ge2}\binom{|S|}{2}\frac{1}{(|S|-1)|S|}\,\hat f(S)\hat g(S)
=\sum_{|S|\ge2}\hat f(S)\hat g(S).\tag*{\qedhere}
\]
\end{proof}

\begin{proof}[Proof of~\cref{thm:main thm-2}]
By~\cref{lem:equiv-submod}, in the supermodular case $\partial_{ij}f,\partial_{ij}g\ge0$ pointwise for all $i\ne j$; in the submodular case both are $\le0$. Since $P_t$ is order-preserving, in either case
\[
\mathbb{E}[\partial_{ij} f\cdot P_t\partial_{ij} g]\ge0
\]
for every $i<j$ and $t>0$. Integrating against the nonnegative kernel $(1-e^{-t})e^{-t}$ and summing over $i<j$ in~\cref{lem:representation_of_high_level_correlation} yields $\sum_{|S|\ge2}\hat f(S)\hat g(S)\ge 0$. Finally,
\[
\Cov(f,g)=\sum_{S\neq\emptyset}\hat f(S)\hat g(S)
=\sum_{i=1}^n \hat f(\{i\})\hat g(\{i\}) + \sum_{|S|\ge2}\hat f(S)\hat g(S)
\ge\sum_{i=1}^n \hat f(\{i\})\hat g(\{i\}),
\]
as claimed.
\end{proof}

\begin{proof}[Proof of Corollary~\ref{cor:L1-influence-main}]
For increasing $h$, we have $\partial_i h\ge0$ pointwise. Hence, by Fact~\ref{fact:fourier expansion of derivative}, $\Inf_i^{(1)}[h]=\mathbb{E}[\partial_i h]=-2\hat h(\{i\})$. Applying this identity to both $f$ and $g$ and using \cref{thm:main thm-2}, we obtain
\[
\Cov(f,g) \ge \sum_{i=1}^n \hat f(\{i\})\hat g(\{i\})=\frac14\sum_{i=1}^n
\Inf_i^{(1)}[f]\Inf_i^{(1)}[g].
\]
If, in addition, $f$ and $g$ are $\{0,1\}$-valued, then $\partial_i f,\partial_i g\in\{0,1\}$ pointwise, so $\Inf_i^{(1)}[\cdot]=\Inf_i[\cdot]$ for both functions. Thus \cref{thm:main thm-1} follows.
\end{proof}

\subsection{A Spectral Friedgut--Kahn--Kalai--Keller Bound}\label{subsec:spectral-FKKK}

We now record a first-order semigroup identity that is useful for the spectral conjecture \eqref{eq:spectral-conj}. The identity is the standard first-order semigroup representation for covariance used, in closely related form, by Keller--Mossel--Sen~\cite{KMS2014correlation}.

\begin{lemma}[First-order semigroup representation]\label{lem:first-order-semigroup-covariance}
For any $f,g:\{0,1\}^n\to\mathbb{R}$,
\begin{equation}\label{eq:first-order-semigroup-covariance}
\Cov(f,g)=\frac14\int_0^\infty e^{-t}\sum_{i=1}^n\tup{\partial_i f,P_t\partial_i g}\,dt.
\end{equation}
\end{lemma}

\begin{proof}
By Facts~\ref{fact:fourier expansion of derivative} and~\ref{fact:heat-fourier},
\[
\tup{\partial_i f,P_t\partial_i g}
=4\sum_{S\ni i} e^{-t(|S|-1)}\hat f(S)\hat g(S).
\]
Multiplying by $e^{-t}$, summing over $i$, and integrating in $t$ gives
\[
\frac14\int_0^\infty e^{-t}\sum_{i=1}^n\tup{\partial_i f,P_t\partial_i g}\,dt
=\sum_{S\neq\emptyset}\left(|S|\int_0^\infty e^{-t|S|}\,dt\right)\hat f(S)\hat g(S)
=\sum_{S\neq\emptyset}\hat f(S)\hat g(S),
\]
which is exactly $\Cov(f,g)$ by Parseval's identity.
\end{proof}

We view $\{0,1\}^n$ as the additive group $\mathbb{F}_2^n$. For two real-valued functions $f,g:\{0,1\}^n\to\mathbb{R}$, their \emph{convolution} is defined by
\[
(f*g)(x):=\mathbb{E}_{y\in\{0,1\}^n}\big[f(y)g(x\oplus y)\big],
\]
where $\oplus$ denotes coordinatewise addition modulo $2$.

\begin{fact}[Fourier transform of convolution; see {\cite[Definition~1.24 and Theorem~1.27]{Ryanbook}}]\label{fact:convolution-fourier}
For any $f,g:\{0,1\}^n\to\mathbb{R}$ and any $S\subseteq[n]$,
\[
\widehat{f*g}(S)=\hat f(S)\hat g(S).
\]
Consequently, $(f*g)(x)=\sum_{S\subseteq[n]}\hat f(S)\hat g(S)\chi_S(x)$.
\end{fact}

\begin{prop}[Spectral-defect representation]\label{prop:spectral-defect-semigroup}
Let $f,g:\{0,1\}^n\to\{0,1\}$ be increasing. Then
\begin{equation}\label{eq:spectral-defect-semigroup}
\Cov(f,g)-4\sum_{S\neq\emptyset}|S|\hat f(S)^2\hat g(S)^2
=
\frac14\int_0^\infty e^{-t}\sum_{i=1}^n
\left(
\tup{\partial_i f,P_t\partial_i g}
-4\|\partial_i(f*g)\|_2^2
\right)dt.
\end{equation}
\end{prop}

\begin{proof}
The covariance part is~\eqref{eq:first-order-semigroup-covariance}. For the quadratic part, by Facts~\ref{fact:fourier expansion of derivative} and~\ref{fact:convolution-fourier}
\[
\partial_i(f*g)
=-2\sum_{S\ni i}\hat f(S)\hat g(S)\chi_{S\setminus\{i\}}.
\]
Thus, by Parseval's identity,
\[
\sum_{i=1}^n\|\partial_i(f*g)\|_2^2
=4\sum_{i=1}^n\sum_{S\ni i}\hat f(S)^2\hat g(S)^2
=4\sum_{S\neq\emptyset}|S|\hat f(S)^2\hat g(S)^2.
\]
Since $\int_0^\infty e^{-t}\,dt=1$, subtracting this identity from~\eqref{eq:first-order-semigroup-covariance} gives \eqref{eq:spectral-defect-semigroup}.
\end{proof}

\begin{lemma}[Derivative of convolution]\label{lem:derivative-convolution-bound}
For every $i\in[n]$,
\[
\partial_i(f*g)=-\frac12\, (\partial_i f* \partial_i g).
\]
Consequently,
\[
4\|\partial_i(f*g)\|_2^2
=
\|\partial_i f*\partial_i g\|_2^2
\le
\|\partial_i f\|_2^2\,\|\partial_i g\|_2^2.
\]
Moreover, equality in the last inequality holds if and only if $\hat f(S)\hat g(T)=0$ for all distinct $S,T\ni i$. Equivalently, either $\partial_i f\equiv0$, or $\partial_i g\equiv0$, or there exists a unique $U\ni i$ such that $\hat f(S)=\hat g(S)=0$ for every $S\ni i$ with $S\neq U$.
\end{lemma}

\begin{proof}
By Fact~\ref{fact:fourier expansion of derivative}, the Fourier coefficient of $\partial_i f$ at $R\subseteq[n]\setminus\{i\}$ is $-2\hat f(R\cup\{i\})$, and similarly for $g$. Hence
\[
\partial_i f*\partial_i g
=4\sum_{S\ni i}\hat f(S)\hat g(S)\chi_{S\setminus\{i\}},
\]
whereas
\[
\partial_i(f*g)
=-2\sum_{S\ni i}\hat f(S)\hat g(S)\chi_{S\setminus\{i\}}.
\]
This proves $\partial_i(f*g)=-\frac12(\partial_i f*\partial_i g)$. Furthermore,
\begin{align*}
\|\partial_i f\|_2^2\|\partial_i g\|_2^2-
\|\partial_i f*\partial_i g\|_2^2
&=16\left(\sum_{S\ni i}\hat f(S)^2\right)
\left(\sum_{T\ni i}\hat g(T)^2\right)
-16\sum_{S\ni i}\hat f(S)^2\hat g(S)^2\\
&=16\sum_{\substack{S,T\ni i\\S\neq T}}\hat f(S)^2\hat g(T)^2\ge0.
\end{align*}
This proves the norm inequality. Equality holds if and only if every summand in the final expression is zero, equivalently $\hat f(S)\hat g(T)=0$ for all distinct $S,T\ni i$. This is the stated support condition.
\end{proof}

\begin{cor}\label{cor:lower-bound-via-noisy-derivatives}\label{cor:fourier-lower-bound}
Let $f,g:\{0,1\}^n\to\{0,1\}$ be increasing. Then
\begin{equation}\label{eq:lower-bound-via-noisy-derivatives}
\Cov(f,g)-4\sum_{S\neq\emptyset}|S|\hat f(S)^2\hat g(S)^2
\ge
\frac14\int_0^\infty e^{-t}\sum_{i=1}^n
\Cov(\partial_i f,P_t\partial_i g)\,dt.
\end{equation}
Consequently,
\begin{equation}\label{eq:fourier-lower-bound}
\Cov(f,g)-4\sum_{S\neq\emptyset}|S|\hat f(S)^2\hat g(S)^2
\ge
\sum_{|S|\ge2}\hat f(S)\hat g(S).
\end{equation}
\end{cor}

\begin{proof}
By Proposition~\ref{prop:spectral-defect-semigroup} and Lemma~\ref{lem:derivative-convolution-bound},
$$\Cov(f,g)-4\sum_{S\neq\emptyset}|S|\hat f(S)^2\hat g(S)^2\ge
\frac14\int_0^\infty e^{-t}\sum_{i=1}^n
\left(
\tup{\partial_i f,P_t\partial_i g}
-\|\partial_i f\|_2^2\|\partial_i g\|_2^2
\right)dt.$$
Since $f$ and $g$ are increasing and Boolean-valued, $\partial_i f,\partial_i g\in\{0,1\}$ pointwise. Hence $\|\partial_i f\|_2^2=\mathbb{E}[\partial_i f]$ and $\|\partial_i g\|_2^2=\mathbb{E}[\partial_i g]$. Moreover, $P_t$ preserves means, and therefore
\[
\tup{\partial_i f,P_t\partial_i g}
-\|\partial_i f\|_2^2\|\partial_i g\|_2^2
=
\mathbb{E}[\partial_i f\cdot P_t\partial_i g]
-\mathbb{E}[\partial_i f]\mathbb{E}[P_t\partial_i g]
=
\Cov(\partial_i f,P_t\partial_i g).
\]
This proves \eqref{eq:lower-bound-via-noisy-derivatives}.

It remains to compute the right-hand side of \eqref{eq:lower-bound-via-noisy-derivatives}. Since $P_t$ acts on the lower-dimensional cube by multiplying the level-$k$ Fourier coefficients by $e^{-tk}$,
\[
\Cov(\partial_i f,P_t\partial_i g)=4\sum_{\substack{S\ni i\\ |S|\ge2}}
e^{-t(|S|-1)}\hat f(S)\hat g(S).
\]
Substituting this into \eqref{eq:lower-bound-via-noisy-derivatives} gives
\[
\begin{aligned}
\frac14\int_0^\infty e^{-t}\sum_{i=1}^n
\Cov(\partial_i f,P_t\partial_i g)\,dt
&=
\sum_{|S|\ge2}
\left(\sum_{i\in S}\int_0^\infty e^{-t|S|}\,dt\right)
\hat f(S)\hat g(S)  \\
&=
\sum_{|S|\ge2}
\left(|S|\cdot\frac1{|S|}\right)
\hat f(S)\hat g(S)=\sum_{|S|\ge2}\hat f(S)\hat g(S).
\end{aligned}
\]
Together with \eqref{eq:lower-bound-via-noisy-derivatives}, this proves \eqref{eq:fourier-lower-bound}.
\end{proof}

\begin{remark}\label{rem:spectral-sign-condition}
Let $f,g:\{0,1\}^n\to\{0,1\}$ be increasing. If
\[
\Cov(\partial_i f,P_t\partial_i g)\ge0
\quad\text{for every } i\in[n]\text{ and every }t\ge0,
\]
then Corollary~\ref{cor:lower-bound-via-noisy-derivatives} gives
\[
\Cov(f,g)\ge4\sum_{S\neq\emptyset}|S|\hat f(S)^2\hat g(S)^2.
\]
Indeed, $P_t$ preserves monotonicity. Thus if $\partial_i f$ and $\partial_i g$ have the same monotonicity type, Harris--Kleitman's inequality applies either directly, or after multiplying both functions by $-1$.
\end{remark}

\begin{proof}[Proof of~Corollary~\ref{cor:supermodular-submodular}]
If $f$ and $g$ are supermodular, then for each $i$ the functions $\partial_i f$ and $\partial_i g$ are increasing in all remaining coordinates. If $f$ and $g$ are submodular, then these derivatives are decreasing in all remaining coordinates. Since $P_t$ preserves the same monotonicity type, Harris--Kleitman's inequality implies
\[
\Cov(\partial_i f,P_t\partial_i g)\ge0
\]
for every $i$ and $t\ge0$. The result follows from the preceding remark.
\end{proof}

\begin{cor}[Diagonal spectral inequality]\label{cor:diagonal-case}
For every increasing $f:\{0,1\}^n\to\{0,1\}$,
\begin{equation}\label{eq:diagonal-spectral}
\Var(f)\ge4\sum_{S\neq\emptyset}|S|\hat f(S)^4.
\end{equation}
\end{cor}

\begin{proof}
Apply Corollary~\ref{cor:lower-bound-via-noisy-derivatives} with $g=f$. For every $i\in[n]$ and $t\ge0$,
\[
\Cov(\partial_i f,P_t\partial_i f)=\sum_{\emptyset\neq R\subseteq[n]\setminus\{i\}}e^{-t|R|}\widehat{\partial_i f}(R)^2\ge0.
\]
Therefore $\Var(f)=\Cov(f,f)\ge4\sum_{S\neq\emptyset}|S|\hat f(S)^4$.
\end{proof}

\section{Upper Bounds for Correlation}\label{sec:upper}

\subsection{Talagrand-Type Upper Bound}
We now turn to an $L^1$-$L^2$ upper bound of Talagrand type for the Level-$\ge2$ contribution. The proof below follows the semigroup interpolation argument in Mossel's lecture notes~\cite[Page 147]{Mossel2020}. The new input here is that Lemma~\ref{lem:representation_of_high_level_correlation} replaces the first-order derivative representation by a second-order derivative representation.
\begin{theorem}\label{thm:Talagrand level-2 bound}
Let $f,g:\{0,1\}^n\to\mathbb{R}$. Then
\begin{equation}
\left|\sum_{|S|\ge2}\hat f(S)\hat g(S)\right|\le \frac{9}{8}\sum_{1\le i<j\le n}
\frac{\|\partial_{ij} f\|_{2}\|\partial_{ij} g\|_{2}}
{1+\log\left(\dfrac{\|\partial_{ij} f\|_{2}\|\partial_{ij} g\|_{2}}{\|\partial_{ij} f\|_{1}\|\partial_{ij} g\|_{1}}\right)}.
\end{equation}
A summand is interpreted as $0$ if $\partial_{ij}f\equiv0$ or $\partial_{ij}g\equiv0$.
\end{theorem}

\begin{proof}
We start from the semigroup representation in \cref{lem:representation_of_high_level_correlation}:
\begin{equation}\label{eq:semigroup representation}
\left|\sum_{|S|\ge2}\hat f(S)\hat g(S)\right|
\le\frac{1}{8}\sum_{1\le i<j\le n}\int_0^\infty (1-e^{-t})e^{-t}\left|\tup{\partial_{ij}f,P_t(\partial_{ij}g)}\right|dt.
\end{equation}
By self-adjointness of the semigroup,
$\tup{\partial_{ij}f,P_t(\partial_{ij}g)}=\tup{P_{t/2}(\partial_{ij}f),P_{t/2}(\partial_{ij}g)}$. Hence, by Cauchy--Schwarz,
\[
\left|\tup{\partial_{ij}f,P_t(\partial_{ij}g)}\right|
\le \|P_{t/2}(\partial_{ij}f)\|_{2}\,\|P_{t/2}(\partial_{ij}g)\|_{2}.
\]

\begin{fact}[Littlewood interpolation]
For $\theta\in(0,1)$, let $p_{\theta},p_1,p_2$ be such that $\frac{1}{p_{\theta}}=\frac{\theta}{p_1}+\frac{1-\theta}{p_2}$. Then
\[
\|h\|_{p_{\theta}}\le \|h\|_{p_1}^{\theta}\|h\|_{p_2}^{1-\theta}.
\]
\end{fact}

We now invoke the standard one-parameter hypercontractive inequality (\cref{thm:hyperc}) and Littlewood interpolation to $h\in\{\partial_{ij}f,\partial_{ij}g\}$ with $p_{\theta}=1+e^{-t}$, $p_1=1$, and $p_2=2$: for every $t\ge0$,
\[
\|P_{t/2} h\|_2 \le\|h\|_{1+e^{-t}}
\le \|h\|_1^{\alpha(t)}\|h\|_2^{1-\alpha(t)},
\qquad \alpha(t)=\frac{1-e^{-t}}{1+e^{-t}}=\tanh\left(\frac{t}{2}\right).
\]
Thus
\[
\begin{aligned}
\left|\tup{\partial_{ij}f,P_t \partial_{ij}g}\right|
&\le \|P_{t/2}\partial_{ij}f\|_2\|P_{t/2}\partial_{ij}g\|_2
\le \|\partial_{ij}f\|_{1+e^{-t}}\|\partial_{ij}g\|_{1+e^{-t}}\\
&\le \|\partial_{ij}f\|_2\|\partial_{ij}g\|_2
\left(\frac{\|\partial_{ij}f\|_2\|\partial_{ij}g\|_2}{\|\partial_{ij}f\|_1\|\partial_{ij}g\|_1}\right)^{-\alpha(t)}.
\end{aligned}
\]
For pairs with $\partial_{ij}f,\partial_{ij}g\not\equiv0$, define $R_{ij}:=\frac{\|\partial_{ij}f\|_2\|\partial_{ij}g\|_2}{\|\partial_{ij}f\|_1\|\partial_{ij}g\|_1}\ge1$. Plugging this into \eqref{eq:semigroup representation} reduces the proof to the one-dimensional kernel estimate
\begin{equation}\label{ineq:intergal}
\mathcal{I}(R):=\int_0^\infty (1-e^{-t})e^{-t} R^{-\tanh(t/2)}dt\le \frac{9}{1+\log R}\qquad(R\ge1).
\end{equation}
Indeed, this estimate immediately gives the desired inequality.

It remains to prove \eqref{ineq:intergal}. Set $u=e^{-t}\in(0,1]$. Since $\tanh(\frac{t}{2})=\frac{1-u}{1+u}$ and $dt=-\frac{du}{u}$,
\[
\mathcal{I}(R)=\int_0^1 (1-u)R^{-\frac{1-u}{1+u}}du.
\]
Let $L:=\log R\ge0$ and split the integral at $u_0:=\frac{1}{1+L}$.

\textbf{Case 1}. If $u\in[0,u_0]$, then every factor in the integrand is at most $1$, and therefore
\[
\mathcal{I}_1:=\int_0^{u_0} (1-u)R^{-\frac{1-u}{1+u}}du\le\int_0^{u_0}1\,du=\frac{1}{1+L}.
\]

\textbf{Case 2}. If $u\in[u_0,1]$, then $\frac{1-u}{1+u}\ge \frac{1-u}{2}$, so
\[
\mathcal{I}_2:=\int_{u_0}^{1} (1-u)R^{-\frac{1-u}{1+u}}du
\le \int_{u_0}^{1} (1-u)e^{-\frac{L}{2}(1-u)}du
\le \int_{0}^{\infty} y e^{-\frac{L}{2}y}dy=\frac{4}{L^2}.
\]
For $L\ge1$, combining these two bounds gives
\[
\mathcal{I}(R)\le \frac{1}{1+L}+\frac{4}{L^2}\le \frac{9}{1+L}.
\]
For $L\in[0,1]$, we have $\mathcal{I}(R)\le \int_0^1(1-u)du=\tfrac12\le \tfrac{9}{1+L}$. This proves \eqref{ineq:intergal}.
\end{proof}

\begin{proof}[Proof of~\cref{thm:main-thm4}]
For increasing $f$ and $g$, Fact~\ref{fact:fourier expansion of derivative} gives $\mathrm{Inf}^{(1)}_i[f]=-2\hat f(\{i\})$ and $\mathrm{Inf}^{(1)}_i[g]=-2\hat g(\{i\})$. Hence
\[
\frac14\sum_{i=1}^n\mathrm{Inf}^{(1)}_i[f]\mathrm{Inf}^{(1)}_i[g]
=\sum_{i=1}^n\hat f(\{i\})\hat g(\{i\}),
\]
and therefore
\[
\Cov(f,g)-\frac14\sum_{i=1}^n\mathrm{Inf}^{(1)}_i[f]\mathrm{Inf}^{(1)}_i[g]
=\sum_{|S|\ge2}\hat f(S)\hat g(S).
\]
Applying \cref{thm:Talagrand level-2 bound} proves \cref{thm:main-thm4}.
\end{proof}

\begin{remark}
With the change of variable $u=\tanh(\frac{t}{2})\in[0,1)$, the kernel estimate can be sharpened, giving the slightly better constant $C=\frac{1+\sqrt{8}}{16}$ in place of $\frac{9}{8}$.
\end{remark}

Inspired by the recent work of Przyby{\l}owski~\cite{przybylowski2024kkl}, we extend the Talagrand-type $L^1$-$L^2$ upper bound to higher Fourier levels. Let $d\ge2$ and $f,g:\{0,1\}^n\to\mathbb R$. For $T=\{i_1,\dots,i_d\}\subseteq[n]$ with $|T|=d$, write $\partial_T f:=\partial_{i_1}\circ \cdots \circ \partial_{i_d}f$. A direct Fourier calculation gives, for every $T$ and $S$,
\[
\partial_T\chi_S=
\begin{cases}
(-2)^d\chi_{S\setminus T},& T\subseteq S,\\
0,& T\not\subseteq S.
\end{cases}
\]
Hence
\[
\tup{\partial_T f,P_t(\partial_T g)}=4^d\sum_{S\supseteq T}e^{-t(|S|-d)}\hat{f}(S)\hat{g}(S).
\]
Summing over $|T|=d$ yields
\[
\sum_{|T|=d}\tup{\partial_T f,P_t(\partial_T g)}
=4^d\sum_{|S|\ge d}\binom{|S|}{d}e^{-t(|S|-d)}\hat{f}(S)\hat{g}(S).
\]
Now integrate against $\frac{d}{4^d}(1-e^{-t})^{d-1}e^{-t}$ and use the beta integral
\[
\int_0^\infty (1-e^{-t})^{d-1}e^{-t}e^{-t(|S|-d)}dt
=\int_0^1 (1-u)^{d-1}u^{|S|-d}du=\frac{(d-1)!(|S|-d)!}{|S|!}.
\]
For $m=|S|\ge d$, the resulting coefficient is
\[
\frac{d}{4^d}\cdot 4^d\binom{m}{d}\frac{(d-1)!(m-d)!}{m!}=1.
\]
Therefore
\begin{equation}\label{eq:level-d-semigroup-representation}
\sum_{|S|\ge d}\hat{f}(S)\hat{g}(S)=\frac{d}{4^d}\sum_{|T|=d}\int_0^\infty(1-e^{-t})^{d-1}e^{-t}
\tup{\partial_T f,P_t(\partial_T g)}dt.
\end{equation}
For $d=2$, this reduces to Lemma~\ref{lem:representation_of_high_level_correlation}. Combining \eqref{eq:level-d-semigroup-representation} with the same
semigroup--hypercontractive interpolation scheme as in~\cref{thm:Talagrand level-2 bound} yields the following higher-level estimate. The details, including the $d$-dependent one-dimensional kernel estimate, are given in Appendix~\ref{app:level-d-upper}.
\begin{cor}\label{cor:level-d-talagrand}
Let $d\ge2$ and let $f,g:\{0,1\}^n\to\mathbb{R}$. Then
\begin{equation}\label{eq:level-d-talagrand}
\left|\sum_{|S|\ge d}\hat f(S)\hat g(S)\right|\le C_d\sum_{T\subseteq[n]:|T|=d}
\frac{\|\partial_T f\|_{2}\|\partial_T g\|_{2}}
{1+\log\left(\dfrac{\|\partial_T f\|_{2}\|\partial_T g\|_{2}}{\|\partial_T f\|_{1}\|\partial_T g\|_{1}}\right)}, \quad C_d=\frac{1+(2^dd!)^{1/d}}{4^d}.
\end{equation}
A summand is interpreted as $0$ if $\partial_Tf\equiv0$ or $\partial_Tg\equiv0$.
\end{cor}
For $d=2$, this gives $C_2=\frac{1+\sqrt8}{16}$, matching the improved constant noted above.

\subsection{Mossel's Two-Function Poincar\'e Inequality and an Iterated Refinement}

We start with a two-function inequality controlling covariance by coordinatewise influences. We include the short proof for completeness.
\begin{lemma}[Mossel's two-function Poincar\'e inequality~\cite{Mossel2020}]\label{lem:two-function-poincare}
For any real-valued functions $f,g:\{0,1\}^n\to\mathbb{R}$,
\[
|\Cov(f,g)|\le \frac{1}{4}\sum_{i=1}^n\sqrt{{\rm Inf}_i[f]{\rm Inf}_i[g]}.
\]
\end{lemma}
\begin{proof}
By Plancherel and Cauchy--Schwarz,
\[
\begin{aligned}
|\Cov(f,g)|
&=\left|\sum_{S\neq\emptyset}\hat{f}(S)\hat{g}(S)\right|
\le\sum_{|S|\ge1}|S|\,\bigl|\hat{f}(S)\bigr|\,\bigl|\hat{g}(S)\bigr|\\
&=\sum_{i=1}^n\sum_{S:i\in S}\bigl|\hat{f}(S)\bigr|\,\bigl|\hat{g}(S)\bigr|
\le \sum_{i=1}^n\sqrt{\sum_{S:i\in S}\bigl|\hat{f}(S)\bigr|^2}\sqrt{\sum_{S:i\in S}\bigl|\hat{g}(S)\bigr|^2}\\
&=\frac{1}{4}\sum_{i=1}^n\sqrt{4\sum_{S:i\in S}\hat{f}(S)^2}\sqrt{4\sum_{S:i\in S}\hat{g}(S)^2}
=\frac{1}{4}\sum_{i=1}^n\sqrt{{\rm Inf}_i[f]{\rm Inf}_i[g]}.
\end{aligned}
\]
Here we used the identity ${\rm Inf}_i[h]=4\sum_{S:i\in S}\hat h(S)^2$.
\end{proof}

We next prove the refinement stated in the introduction.
\begin{proof}[Proof of~\cref{thm:stronger two poincare}]
Covariance is invariant under adding constants. Moreover, the right-hand side of
\eqref{eq:iterated-level-one-refinement} only involves non-empty Fourier coefficients and influences, and hence is also invariant under adding constants. Thus we may assume that $\mathbb{E}[f]=\mathbb{E}[g]=0$.

Fix $J\subseteq[n]$. Write $x=(x_J,y)$, where
$y\in\{0,1\}^{[n]\setminus J}$. For each $A\subseteq J$, define functions
$f_A,g_A:\{0,1\}^{[n]\setminus J}\to\mathbb{R}$ by
\[
f_A(y):=\sum_{B\subseteq[n]\setminus J}\hat f(A\cup B)\chi_B(y),
\quad
g_A(y):=\sum_{B\subseteq[n]\setminus J}\hat g(A\cup B)\chi_B(y).
\]
Then we have $f(x)=\sum_{A\subseteq J}\chi_A(x_J)f_A(y)$ and $g(x)=\sum_{A\subseteq J}\chi_A(x_J)g_A(y)$.
By orthogonality in the coordinates of $J$, and since $\mathbb{E}[f]=\mathbb{E}[g]=0$,
\[
\Cov(f,g)=\tup{f,g}=\sum_{A\subseteq J}\tup{f_A,g_A},
\]
where the inner products on the right are taken over the cube
$\{0,1\}^{[n]\setminus J}$. For each $A\subseteq J$, we have $\mathbb{E}[f_A]=\hat f(A)$ and $\mathbb{E}[g_A]=\hat g(A)$. In particular, $\mathbb{E}[f_\emptyset]=\mathbb{E}[g_\emptyset]=0$. Hence
\[
\tup{f_A,g_A}=\Cov(f_A,g_A)+\hat f(A)\hat g(A),
\]
and therefore
\[
|\Cov(f,g)|\le \sum_{A\subseteq J}|\tup{f_A,g_A}|\le
\sum_{\emptyset\ne A\subseteq J}|\hat f(A)\hat g(A)|+\sum_{A\subseteq J}|\Cov(f_A,g_A)|.
\]

We now apply Mossel's two-function Poincar\'e inequality (Lemma~\ref{lem:two-function-poincare}) to each pair $(f_A,g_A)$ on the cube $\{0,1\}^{[n]\setminus J}$. This gives
\[
|\Cov(f_A,g_A)|\le\frac14\sum_{j\in[n]\setminus J} \sqrt{{\rm Inf}_j[f_A]{\rm Inf}_j[g_A]}.
\]
Consequently,
\[
\sum_{A\subseteq J}|\Cov(f_A,g_A)|\le\frac14\sum_{j\in[n]\setminus J}\sum_{A\subseteq J}\sqrt{{\rm Inf}_j[f_A]{\rm Inf}_j[g_A]}\le\frac14\sum_{j\in[n]\setminus J}\sqrt{\sum_{A\subseteq J}{\rm Inf}_j[f_A]}\sqrt{\sum_{A\subseteq J}{\rm Inf}_j[g_A]},
\]
by Cauchy--Schwarz. For every $j\in[n]\setminus J$, a direct Fourier calculation gives
\[
\sum_{A\subseteq J}{\rm Inf}_j[f_A]=4\sum_{A\subseteq J}
\sum_{\substack{B\subseteq[n]\setminus J\\ j\in B}}\hat f(A\cup B)^2=
4\sum_{S:j\in S}\hat f(S)^2={\rm Inf}_j[f],
\]
and similarly $\sum_{A\subseteq J}{\rm Inf}_j[g_A]={\rm Inf}_j[g]$. Substituting this into the previous estimate yields
\[
|\Cov(f,g)|\le\sum_{\emptyset\ne A\subseteq J}|\hat f(A)\hat g(A)|+\frac14\sum_{j\in[n]\setminus J}\sqrt{{\rm Inf}_j[f]{\rm Inf}_j[g]},
\]
which proves the theorem.
\end{proof}

\section{Concluding Remarks}\label{sec:reverse}
Reverse hypercontractivity suggests a natural route to strengthening correlation lower bounds. We begin by recalling Borell's form on the hypercube~\cite{Borel1982,MORSS2006}. 
\begin{theorem}\label{thm:reverse hypercontra}
Let $f,g:\{0,1\}^n\to[0,\infty)$. Then for any $p,q\in(0,1)$ such that $e^{-2t}\le (1-p)(1-q)$,
\[
\tup{f,P_tg}\ge \|f\|_p\|g\|_q.
\]
\end{theorem}
We now explain how this yields a quantitative correlation bound under the second-difference sign conditions. Assume that $f,g:\{0,1\}^n\to\mathbb{R}$ satisfy $\partial_{ij}f,\partial_{ij}g\ge0$ pointwise. Fix any $\theta\in(0,1)$ and set $p=q=1-\theta$. If $t\ge t_0:=\log(1/\theta)$, then $e^{-2t}\le e^{-2t_0}=\theta^2=(1-p)(1-q)$, so \cref{thm:reverse hypercontra} gives
\[
\tup{\partial_{ij}f, P_t\partial_{ij}g}\ge\|\partial_{ij}f\|_{1-\theta}\|\partial_{ij}g\|_{1-\theta}, \quad t\ge t_0.
\]
In the analogous case $\partial_{ij}f,\partial_{ij}g\le0$ pointwise, apply \cref{thm:reverse hypercontra} to $-\partial_{ij}f$ and $-\partial_{ij}g$. Plugging the resulting bound into the heat-semigroup representation \cref{lem:representation_of_high_level_correlation} and integrating over $[t_0,\infty)$ yields
\[
\sum_{|S|\ge 2}\hat f(S)\hat g(S)
\ge\frac{1}{8}\sum_{1\le i<j\le n}\int_{t_0}^\infty \left(1-e^{-t}\right)e^{-t}\|\partial_{ij}f\|_{1-\theta}\|\partial_{ij}g\|_{1-\theta}dt.
\]
Since
\[
\int_{t_0}^\infty \left(1-e^{-t}\right)e^{-t}dt
=\int_0^{e^{-t_0}}(1-u)du
=\theta-\frac{1}{2}\theta^2,
\]
we obtain the dimension-free lower bound
\begin{equation}\label{ineq:strong lower bound}
 \sum_{|S|\ge 2}\hat f(S)\hat g(S)
\ge c(\theta)\sum_{1\le i<j\le n}\|\partial_{ij}f\|_{1-\theta}\|\partial_{ij}g\|_{1-\theta},
\end{equation}
where $c(\theta):=\frac{\theta-\theta^2/2}{8}\in (0,1)$ for $\theta\in(0,1)$. Combining \eqref{ineq:strong lower bound} with
\[
\Cov(f,g)=\sum_{i=1}^n\hat{f}(\{i\})\hat{g}(\{i\})+\sum_{|S|\ge2}\hat{f}(S)\hat{g}(S),
\]
we obtain the strengthened lower bound
\begin{equation}
\Cov(f,g)\ge \sum_{i=1}^n\hat{f}(\{i\})\hat{g}(\{i\})+c(\theta)\sum_{1\le i<j\le n}\|\partial_{ij}f\|_{1-\theta}\|\partial_{ij}g\|_{1-\theta}.
\end{equation}
Moreover, since every term on the right-hand side of \eqref{ineq:strong lower bound} is nonnegative, we have the following exact criterion for vanishing Level-$\ge2$ weight.
\begin{cor}\label{cor:structure information of zero level2 weight}
Assume that $f,g:\{0,1\}^n\to\mathbb{R}$ satisfy $\partial_{ij}f,\partial_{ij} g\ge0$ pointwise for all $i\ne j$ (or $\partial_{ij}f,\partial_{ij} g\le0$ pointwise for all $i\ne j$). Then
\[
\sum_{|S|\ge2}\hat{f}(S)\hat{g}(S)=0
\Longleftrightarrow
\|\partial_{ij}f\|_{1-\theta}\cdot\|\partial_{ij}g\|_{1-\theta}=0
\ \text{for every }1\le i<j\le n,
\]
equivalently, for each pair $(i,j)$ at least one of $\partial_{ij}f$ or $\partial_{ij}g$ is identically zero.
\end{cor}

Without sign information on the second differences, the representation in \cref{lem:representation_of_high_level_correlation} no longer has a nonnegative integrand, and reverse hypercontractivity cannot be applied directly. Equivalently, the flipped-difference operator $D_i f(x):=f(x)-f(x\oplus e_i)$ satisfies $D_{ij}f(x)=(-1)^{x_i+x_j}\partial_{ij}f(x)$ for $i\ne j$; the multiplicative character $(-1)^{x_i+x_j}$ introduces an oscillating sign. This is the same obstruction that appears in semigroup approaches to Talagrand-type correlation inequalities without additional sign structure; compare the proof strategy in~\cite[Proof of Theorem~3.1]{KMS2014correlation}.
\section*{Acknowledgements}

The authors would like to thank the anonymous referees for a careful reading of the manuscript and for many valuable comments and suggestions, which significantly improved the quality of the paper.

This work was initiated during the 3$^{\text{rd}}$ ECOPRO Student Research Program at the Institute for Basic Science (IBS) in the summer of 2025. We are grateful to Prof.~Hong Liu for providing this research opportunity. Fan Chang thanks his advisor, Prof.~Lei Yu, for valuable discussions at the early stage of this project. He also thanks Prof.~Haonan Zhang for introducing him to topics related to Markov semigroups and for many helpful early-stage discussions during his visit to the Institute for Advanced Study in Mathematics at Harbin Institute of Technology in the summer of 2024.

\bibliographystyle{abbrv}
\bibliography{reference}

\appendix

\section{The higher-level $L^1$-$L^2$ estimate}\label{app:level-d-upper}

\begin{proof}[Proof of \cref{cor:level-d-talagrand}]
Starting from~\eqref{eq:level-d-semigroup-representation}, we get
\[
\left|\sum_{|S|\ge d}\hat f(S)\hat g(S)\right|\le\frac{d}{4^d}\sum_{|T|=d}\int_0^\infty(1-e^{-t})^{d-1}e^{-t}\left|\tup{\partial_T f,P_t(\partial_T g)}\right|dt.
\]
Fix $T$ with $\partial_T f\not\equiv0$ and $\partial_T g\not\equiv0$.  By self-adjointness of $P_t$, Cauchy--Schwarz, hypercontractivity, and Littlewood interpolation,
\[
\begin{aligned}
\left|\tup{\partial_T f,P_t\partial_T g}\right|&=\left|\tup{P_{t/2}\partial_T f,P_{t/2}\partial_T g}\right|\le\|P_{t/2}\partial_T f\|_2\|P_{t/2}\partial_T g\|_2 \\
&\le\|\partial_T f\|_{1+e^{-t}}\|\partial_T g\|_{1+e^{-t}}\le\|\partial_T f\|_2\|\partial_T g\|_2 \left(\frac{\|\partial_T f\|_2\|\partial_T g\|_2}{\|\partial_T f\|_1\|\partial_T g\|_1}\right)^{-\tanh(t/2)},
\end{aligned}
\]
Write $R_T:=\frac{\|\partial_T f\|_2\|\partial_T g\|_2}{\|\partial_T f\|_1\|\partial_T g\|_1}\ge 1$. Thus it remains to bound, for $R\ge1$,
\[
\mathcal{I}_d(R):=\int_0^\infty (1-e^{-t})^{d-1}e^{-t}R^{-\tanh(t/2)}dt.
\]
Put $u=e^{-t}$ and $L=\log R$. Then
\[
\mathcal{I}_d(R)=\int_0^1(1-u)^{d-1}R^{-\frac{1-u}{1+u}} du.
\]
First,
\begin{equation}\label{ineq:first case}
\mathcal{I}_d(R)\le\int_0^1(1-u)^{d-1}\,du=\frac1d.
\end{equation}
Second, with $y=1-u$, and using $\frac{y}{2-y}\ge \frac{y}{2}$ for $0\le y\le1$, we get, for $L>0$,
\[
\mathcal{I}_d(R)=\int_0^1y^{d-1}\exp\left(-L\frac{y}{2-y}\right)\,dy\le
\int_0^\infty y^{d-1}e^{-\frac{Ly}{2}}\,dy=\frac{2^d\cdot d!}{dL^d}.
\]
Combining the two estimates gives
\[
\mathcal{I}_d(R)\le\frac{1+(2^d d!)^{1/d}}{d(1+L)}.
\]
Indeed, if $L\le (2^d d!)^{1/d}$, this follows from~\eqref{ineq:first case}. If $L>(2^d d!)^{1/d}$, write $L=a\cdot (2^d d!)^{1/d}$ with $a>1$. Then $1+L=1+a\cdot (2^d d!)^{1/d}\le (1+(2^d d!)^{1/d})a^d$ and hence
\[
\frac{2^d d!}{L^d}=\frac1{a^d}\le\frac{1+(2^d d!)^{1/d}}{1+L}.
\]
This proves the claimed kernel estimate. Substituting this estimate with $R=R_T$ yields~\eqref{eq:level-d-talagrand} with $C_d=\frac{1+(2^d d!)^{1/d}}{4^d}$.
\end{proof}

\end{document}